\documentclass{emsprocart}
\usepackage{srcltx}
\usepackage{amssymb}
\usepackage{amscd}
\usepackage[all,cmtip]{xy}
\usepackage{rotating}
\usepackage{hyperref}

\usepackage{makeidx} 
\makeindex

\newcommand{\GG}{\mathbf{G}} 
\newcommand{\PP}{\mathbf{P}} \newcommand{\QQ}{\mathbf{Q}} 
\newcommand{\ZZ}{\mathbf{Z}} 
 
\newcommand{\One}{\mathbf{1}}

\newcommand{\id}{\mathrm{id}}

\newcommand{\sep}{\mathrm{sep}}

\DeclareMathOperator{\Frob}{Frob} 
\DeclareMathOperator{\Ext}{Ext} 
 
\DeclareMathOperator{\Lie}{Lie}

\DeclareMathOperator{\Res}{Res}

\DeclareMathOperator{\Gal}{Gal}

\DeclareMathOperator{\Hom}{Hom}

\DeclareMathOperator{\Aut}{Aut}

\newcommand{\Ho}{{\mathcal{H}}}

\renewcommand{\th}{{\theta}}

\newcommand\dash{\nobreakdash-\hspace{0pt}}

\newcommand{\comment}[1]{}

\newcommand{\llb}{[\mspace{-1mu}[}  
\newcommand{\rrb}{]\mspace{-1mu}]}  
\newcommand{\llp}{(\mspace{-2mu}(}  
\newcommand{\rrp}{)\mspace{-2mu})}  

\theoremstyle{plain}
\newtheorem{introtheorem}{Theorem}
\newtheorem{theorem}{Theorem}
\newtheorem{corollary}{Corollary}
\newtheorem{lemma}{Lemma}
\newtheorem{proposition}{Proposition}
\theoremstyle{definition}
\newtheorem{definition}{Definition}
\newtheorem{example}{Example}
\newtheorem{remark}{Remark}

\contact[lenny.taelman@gmail.com]{Lenny Taelman, Mathematisch Instituut, Universiteit Leiden, Postbus 9512, 2312 MB Leiden, the Netherlands}

\title[$1$-$t$-Motifs]{$1$-$t$-Motifs}

\author[Lenny Taelman]

\begin{document}

\begin{abstract}
  We show that the module of rational points on an abelian $t$\dash module $ E $ is
  canonically isomorphic with the module $\Ext^1( M_E, K[t] ) $ of extensions of the
  trivial $t$-motif $ K[t] $ by the $t$-motif $ M_E $ associated with $E$. This generalizes
  prior results of Anderson and Thakur, Papanikolas and Ramachandran, and Woo.
  
  In case $ E $ is uniformizable then we show that this extension module is canonically
  isomorphic with the corresponding extension module of Pink-Hodge structures. 
  
  This situation is formally very similar to Deligne's theory of 1-motifs and
  we have tried to build up the theory in a way that makes this analogy as clear as possible.
\end{abstract}

\maketitle

\setcounter{tocdepth}{1}
\tableofcontents

\section{Introduction \& statement of the main results}\label{statements}

Anderson and Thakur \cite{Anderson90} have shown that the group of $ K $\dash rational points on the $n$\dash th tensor power of the Carlitz module \index{Carlitz module} \index{tensor power of the Carlitz module} is canonically isomorphic with the group of extensions 
\index{extension group} of the corresponding $t$-motif $C^{\otimes n}$ over $ K$ by the trivial
$t$\dash motif $K[t]$, \index{trivial $t$-motive} and Papanikolas and Ramachandran \cite{Papanikolas03} and Woo \cite{Woo95} have a similar result for points on Drinfeld modules. \index{Drinfeld module}

We provide a generalization of these results.

The present treatment is less dependent on calculations and more visibly functorial than the prior ones in \cite{Anderson90} and \cite{Papanikolas03}. In fact, this paper grew out of an attempt to understand these results, and it was written with the belief that with an approach which is sufficiently intrinsic and functorial the results would follow in their full generality at once, without extra effort.

We now summarize the main results of this paper.

Fix a finite field $ k $ of $ q $ elements and a field $ K $ containing $ k $. Consider the skew polynomial ring $ K[\sigma] $ \index{skew polynomial ring} whose elements are polynomials $ \sum_i x_i\sigma^i $ and where multiplication is defined by the rule $ \sigma x = x^q \sigma $ for all $ x \in K $. Note that $ K[\sigma] $ is the ring of endomorphisms of the $ k $-vector space scheme $ \GG_{a,K} $, in particular, if $ q $ is a prime number then $ K[\sigma] $ is the endomorphism ring of the additive group scheme over $ K $.  Note that $ k $ is central in $ K[\sigma] $. Denote by $ K[\sigma,t] $ the ring 
$ k[t] \otimes_k K[\sigma] $. Elements of this ring can be identified with polynomials in $ t $ and $ \sigma $ over $ K $ and multiplication satisfies $ t\sigma=\sigma t$, $ tx = xt $ and
$ \sigma x = x^q \sigma $ for all $ x $ in $ K $. 

The field $ K $ is naturally a left $ K[\sigma] $\dash module through $ \sigma x= x^q $ and similarly
the ring $ K[t] $ is naturally a left $ K[t,\sigma] $\dash module.

\begin{introtheorem}\label{thm1} Let $ M $ be a left $ K[t,\sigma] $\dash module which is free and finitely
generated both as a $ K[t] $\dash module and as a $ K[\sigma] $\dash module. There is an
isomorphism of $ k[t] $-modules
\[
	 \Hom_{K[\sigma]}( M, K ) \longrightarrow \Ext^1_{K[t,\sigma]}( M, K[t] ),
\]
functorially in $ M $ and compatible with field extensions $ K \to K' $.
\end{introtheorem}

The proof of this theorem will be given in section \S \ref{algebraictheory}.

Our main interest in this theorem lies in the case where $ M $ is an (effective) $ t $\dash motif.
Assume that a $ k $-algebra homomorphism $ \i : k[t] \to K $ has been fixed (the ``structure homomorphism''), and that relative to this homomorphism $ M $ is an an effective $t$-motif which is finitely generated over $ K[\sigma] $ (\emph{see} \P \ref{tmotifs} for the definition). Then $ \Hom_{K[\sigma]}( M, K ) $ is the module $E(K)$ of $ K $\dash valued points of the abelian
$ t $\dash module $ E $ associated to $ M $ (\emph{see} \P \ref{tmodules}), $ M $ satisfies the hypothesis of Theorem \ref{thm1} and we obtain a canonical isomorphism 
\begin{equation}\label{basis}
	 E( K ) = \Ext^1_{K[t,\sigma]}( M, K[t] ).
\end{equation}
For $ M $ a tensor power of the Carlitz module this was proven by Anderson and Thakur
\cite{Anderson90}, and for $ M $ a Drinfeld module a similar result was shown by 
Papanikolas and Ramachandran \cite{Papanikolas03} (\emph{see} the end of
section \S \ref{algebraictheory} for the precise relation between their results and ours.) 
The isomorphism (\ref{basis}) should be compared with the canonical isomorphism
\[
	A^\vee( L ) = \Ext^1( A, \GG_m )
\]
for an abelian variety $ A $ over a field $ L $. Note however that the collection of
$t$\dash motifs $ M $ to which (\ref{basis}) applies is closed under tensor product, very much unlike the analogous situation with Abelian varieties.

The isomorphism (\ref{basis}) allows us to work with extensions of $t$-motifs like $ M $ by
Artin $t$-motifs ($t$-motifs that are trivialised by a finite separable extension of $ K $, in other words: $t$-motifs corresponding to \emph{finite} image Galois representations) in much the same way as one uses Deligne's theory of $1$-motifs \index{$1$-motives} to work with mixed motifs \index{mixed motives} of weights $-1$ and $0$. 

We call a \emph{1-$t$-module} a triple  $ ( X, E, u ) $ consisting of
\begin{enumerate}
\item a free and finitely generated $ k[ t ] $-module
	$ X $ equipped with a continuous action of
	$ \Gal(K^\sep/K) $;
\item an abelian $ t $-module $ E $ over $ K $;
\item a $ \Gal(K^\sep/K) $-equivariant homomorphism of $ k[ t ] $-modules
	$ u : X \to E( K^\sep ) $.
\end{enumerate}
A \emph{1-$t$-motif} is an effective $t$\dash motif $\tilde{M}$ that is an extension of an effective $t$-motif $ M $ which is finitely generated as a $ K[\sigma] $\dash module by an Artin  $t$-motif $ V $.
In \S \ref{algebraictheory} we show how the canonical isomorphism (\ref{basis}) implies that the category of $1$\dash $t$\dash modules is anti-equivalent with the category of $1$\dash $t$\dash motifs. Under this equivalence $ M $ is determined by $ E $, $ V $ by $ X $ and the extension class of $ \tilde{M} $ by $ u $, via (\ref{basis}). 

\bigskip

The second part of the paper provides an analytic interpretation of the extension group
$ \Ext^1(M,K[t]) $ in terms of Hodge structures of $t$-motifs. We need to assume 
that $ \i : k[t] \to K $ is injective ($ K $ has ``generic characteristic''), that
$ K $ is a local field with absolute value $ | \cdot | $ and that $ |\i(t)| > 1 $ ($ K $ has the ``$\infty$-adic topology''). Denote by $ C $ the completion of the algebraic closure of $ K $. 

Let $ M $ be an effective $ t $-motif over $ K $ which is finitely generated as
a $ K[\sigma] $\dash module and denote by $ E $ the associated abelian $ t $-module. Denote by
$ \Lie_E(C) $ the tangent space at zero to $ E_C $. This is naturally a $ C $\dash vector space
and by functoriality it is also a $ k[t] $-module. There is a unique holomorphic map
\[
	\exp_E : \Lie_E ( C ) \longrightarrow E( C ) 
\]
which is $ k[t] $-linear and tangent to the identity map $ \id : E( C ) \to E( C) $. The kernel of $\exp_E$ is a lattice, that is, a discrete and finitely generated $ k[t] $\dash submodule.
We say that  $ E $ is \emph{uniformizable} if $ \exp_E $ is surjective. This is equivalent 
with $ M $ being \emph{analytically trivial} (\emph{see} \S \ref{unifandhodge} for the definition.) Also the ``unit'' effective $ t $-motif $ K[t] $ (which does not correspond to an abelian $t$-module since it is not finitely generated as a $ K[\sigma] $\dash module) is analytically trivial and any extension of one analytically trivial effective $ t $-motif by another is analytically trivial (\emph{see} Proposition \ref{antrivext}.)

In \S \ref{unifandhodge} we define (a primitive version of) the Hodge structure $ H(M) $ of an analytically trivial effective $t$\dash motif $ M $, following Pink \cite{Pink97}. Also we relate the uniformization of an abelian $t$-module $ E $ to the Hodge structure $H( M )$ of the corresponding $t$\dash motif $ M $.

In \S \ref{transtheory} we show:

\begin{introtheorem} Let $ M $ be an effective $ t $-motif over $ K $, finitely generated 
over $ K[\sigma]$ and analytically trivial. The natural map
\[
	 \Ext^1_{C[t,\sigma]}( M_C, C[t] ) \longrightarrow \Ext^1( H(M_C), H(C[t]) ) 
\]
that maps the class of an extension of $ t $-motifs to the class of its Hodge structure
is an isomorphism of $k[t]$-modules.
\end{introtheorem}

This is an analogue of Deligne's theorem \cite[\S 10]{Deligne74b} on the equivalence between 1\dash motifs and Hodge structures of type $\{(0,0),(0,-1),(-1,0),(-1,-1)\}$, with the
exception that we do not say anything about which Hodge structures occur as $ H( M ) $ for some
analytically trivial effective $ t $-motif $ M $, finitely generated over $ K[\sigma] $. This appears to be a difficult problem.  

\smallskip

{\emph{Acknowledgment.} The author would like to thank the anonymous referee for pointing out a mistake in a previous version of Proposition \ref{PR}.}

\section{Duality for torsion modules over $ k\llb z \rrb $}
\label{torsionduality}

Let $ k $ be a field and $ k\llb z \rrb $ the ring
of power series in a single variable $ z $ over $ k $. This section recalls some easy facts about finitely generated torsion modules over $ k\llb z \rrb $. Not only will we use some of these facts later on, but we will also carry out a number of constructions that are directly inspired by them. 

Denote by $ k\llp z \rrp $ the field of Laurent series in $ z $.
If $ \alpha : M \to N $ is a morphism of $ k\llb z \rrb $-modules,
then we write $ \alpha_\circ $ for the induced map
$ M \otimes k\llp z \rrp \to N \otimes k\llp z \rrp $.
We denote by $ \Res_{z=0} $ the residue map from $ k\llp z \rrp $ to $ k $ that
maps a Laurent series to its coefficient of $ z^{-1} $. (So we have silently identified
$ \hat{\Omega}_{k\llp z \rrp/k} $ with $ k\llp z \rrp $ by choosing the generator $dz$.)

Let 
\begin{equation}\label{short}
	0
	\longrightarrow
	F
	\overset{ \alpha }{ \longrightarrow }
	G
	\overset{ \pi }{ \longrightarrow }
	T
	\longrightarrow
	0
\end{equation}
be a short exact sequence of finitely generated
$ k\llb z \rrb $-modules, with $ T $ a torsion module and $ F $
and $ G $ free. It follows that $ \alpha_\circ $ is an isomorphism. This
short exact sequence induces a short exact sequence of
$ k\llb z \rrb $-modules
\begin{equation}\label{dualshort}
	0
	\longrightarrow 
	\Hom_{k\llb z \rrb} ( G, k\llb z \rrb )
	\overset{ \alpha^t }{ \longrightarrow }
	\Hom_{k\llb z \rrb} ( F, k\llb z \rrb )
	\longrightarrow
	\Hom_k( T, k )
	\longrightarrow
	0
\end{equation}
where the surjection is given by
\[
	\varphi
	\longmapsto 
	\left[
		\pi( g ) \mapsto 
		\Res_{z=0} \varphi_\circ( \alpha_\circ^{-1}( g ) )
	\right].
\]

The short exact sequence (\ref{dualshort}) can be identified with
\[
	0
	\longrightarrow 
	\Hom_{k\llb z \rrb} ( G, k\llb z \rrb )
	\longrightarrow 
	\Hom_{k\llb z \rrb} ( F, k\llb z \rrb )
	\longrightarrow
	\Ext^1_{k\llb z \rrb} ( T, k\llb z \rrb )
	\longrightarrow
	0,
\]
the long exact sequence of cohomology obtained by applying $ \Hom( -, k\llb z \rrb ) $ to (\ref{short}). Also, the residue map yields an isomorphism of $ k\llb z \rrb $-modules
\[
	 \Hom_{k\llb z \rrb}( T, k\llp z \rrp/k\llb z \rrb )
	 \overset{\sim}{\longrightarrow}
	 \Hom_k( T, k )\colon f \mapsto \Res_{z=0} \circ f.
\]
Alternatively, one can directly obtain an isomorphism
\[
	\Hom_{k\llb z \rrb}( T, k\llp z \rrp/k\llb z \rrb ) \overset{\sim}{\longrightarrow} 
	\Ext^1_{k\llb z \rrb} ( T, k\llb z \rrb )
\]
by applying $\Hom( T, - ) $ to the injective resolution
\[
	0 \longrightarrow k\llb z \rrb \longrightarrow k\llp z \rrp
		\longrightarrow k\llp z \rrp/k\llb z \rrb \longrightarrow 0
\]
of $k\llb z \rrb$.

\section{Effective $t$-motifs and abelian $t$-modules}\label{sectmottmod}

In this section we define effective $t$\dash motifs, abelian $t$\dash modules and
recall Anderson's correspondence \cite{Anderson86} between certain
effective $t$-motifs and abelian $t$-modules. \index{effective $t$-motive}
\index{abelian $t$-modules}

We fix a finite field $ k $ of $ q $ elements. 

\subsection{Effective $ t $-motifs}\label{tmotifs}

Let $ R $ be a commutative ring containing $ k $ and $ R[ t ] $ the polynomial ring in
one variable $ t $ over $ R $. Denote by $ \tau : R[t]\to R[t] $ the ring endomorphism
that restricts to the $q$\dash th power Frobenius endomorphism on $ R $ and that fixes $ t $.
If $ M $ is an $ R[t] $\dash module then we define the pull-back of $ M $ along $ \tau $
as follows:
\[
	M' := \tau^\ast M = M \otimes_{R[t],\tau} R[t].
\]
We give $ M' $ the structure of an $ R[t] $\dash module through the second factor in the tensor product. There is a canonical $k[t]$\dash linear map from $ M $ to $ M' $, which we denote by $ \tau $. It is given by
\[
	\tau: M \to M' : m \mapsto m\otimes 1.
\]

\begin{definition}
A \emph{$ \sigma $-module} over $ R[t] $ is a pair $ ( M, \sigma ) $ of an
$ R[t] $-module $ M $ and an $ R[t] $\dash linear map $ \sigma: M' \to M $.
A morphism of $ \sigma $\dash modules is a morphism $ f: M_1 \to M_2 $ of $ R[t] $\dash modules
such that $ f \circ \sigma_1 = \sigma_2 \circ f' $. 
\end{definition}

We will usually suppress the $ \sigma $ from the notation and write 
$ M $ for a $ \sigma $-module $ ( M, \sigma ) $.

The category of $ \sigma $-modules is abelian and $ k[t] $\dash linear. We denote the
$ k[t] $\dash module of morphisms $ M_1 \to M_2 $ by
\[
	\Hom_{R[t],\sigma}\!\big( M_1, M_2 \big).
\]

\begin{remark} Let $ R[\sigma] $ denote the ring whose elements are polynomials 
$\sum x_i \sigma^i $ in $ \sigma $
but where multiplication is defined through $ \sigma x = x^q \sigma $ for all $ x \in R $. Note 
that $ k $ is central in $ R[\sigma] $. Write $ R[t,\sigma] $ for the tensor product
$ k[t] \otimes_k R[\sigma] $. Let $ ( M, \sigma_M ) $ be a $\sigma$\dash module over $ R[t] $.
Then $ M $ has naturally the structure of a left $ R[t,\sigma] $\dash module through
$ \sigma m := \sigma_M( \tau(m) ) $. In fact, this construction defines an isomorphism
of categories between the category of $ \sigma $\dash modules over $ R[t] $ and the category
of left modules over $ R[t,\sigma] $.  We will use the language of $ \sigma $\dash modules
and the language of $ R[t,\sigma ] $\dash modules interchangeably, each time
choosing the most convenient for the task at hand. There is considerable abuse of notation
in denoting the linear map $ \sigma : M' \to M $ and the semi-linear action of
$ \sigma \in R[t,\sigma] $ on $ M $ by the same symbol $\sigma$. We hope this will not lead to
confusion.
\end{remark}

If $ ( M_1, \sigma_1 ) $ and $ ( M_2, \sigma_2 ) $ are $ \sigma $-modules
then we define their tensor product to be the $ \sigma $-module
$ ( M_1 \otimes_{R[t]} M_2, \sigma_1 \otimes \sigma_2 ) $.
Similarly one can define symmetric and exterior powers. In particular,
given a $ \sigma $-module $ M $ whose underlying module is locally free
of some constant rank $ r $ then one can consider its determinant
$ \det( M ) := \wedge^r M $ which is a $ \sigma $-module that is locally free of rank one.

If $ R \to S $ is a $ k $-algebra homomorphism and $ M $ a
$ \sigma $-module over $ R $ then we denote by $ M_S $ the
$ \sigma $-module over $ S $ obtained by extension of scalars.

Now let $ K $ be a field containing $ k $ and fix a $ k $-algebra homomorphism $ \i : k[t] \to K $. We denote the image of $ t $ by $ \th \in K $. One should think of $ \i $ as an analogue to the canonical homomorphism from $ \ZZ $ to any commutative ring. From now on we shall always consider the field $ K $ as a $ k[t] $\dash algebra, so we will silently consider the distinguished element $ \i(t)=\th \in K $ to be part of the data when referring to ``$K$''.

\begin{definition}
An \emph{effective $ t $-motif\footnote{
The terminology used here is that of \cite{Taelman09a}.  What is called a ``$t$-motive'' in
\cite{Anderson86} would be an ``effective $t$\dash motif that is finitely generated as a 
$K[\sigma]$\dash module'' in our language.} of rank $ r $} over $ K $ is a
$ \sigma $-module $ M $ over $ K $ whose underlying module is free of rank $ r $ and
such that the cokernel of the linear map $ \sigma : M' \to M $ is
annihilated by some power of $ t-\th $.
\end{definition}

The condition on $ \sigma $ is equivalent with the condition that 
$ \det ( M ) $ is isomorphic with the $ \sigma $-module
$ ( K[t]e, \tau e \mapsto \alpha(t-\th)^ne ) $ for some 
$ \alpha \in K^\times $ and $ n \geq 0 $.

If $ \lambda $ is a maximal ideal of $ k[t] $ then we write $ k[t]_\lambda $ for the
$ \lambda $-adic completion of $ k[t] $ and $ k(t)_\lambda $ for its quotient field.
For every $ \lambda$ the $ k[t]_\lambda $\dash module  
\[
	T_\lambda( M ) :=
	\left( M \otimes_{K[t]} K^\sep[t]_\lambda \right)^{\sigma\tau}
\]
of $\sigma\tau$\dash invariants carries a continuous action of
$ \Gal( K^\sep / K ) $ and
\[
	V_\lambda( M ) :=
	T_\lambda( M ) \otimes_{ k[t]_\lambda } k(t)_\lambda
\]
is a $ k(t)_\lambda $-vector space with a continuous action
of $ \Gal( K^\sep / K ) $.

The following proposition gives some justification for the ``motivic'' terminology (which is Anderson's), but we stress that we are merely dealing with an \emph{analogy}; there is no 
known direct relation with any kind of algebro-geometric motifs.

\begin{proposition}[Thm 3.3 of \cite{Gardeyn01}]\label{strcmp}
Assume that $ \i $ is injective and that $ K $ is finite over its subfield $ k(\th) $.
Let $ M $ be an effective $ t $-motif over $ K $ of rank $ r $.
Then $ \dim V_\lambda( M ) = r $ for all $ \lambda $. Moreover,
there exists a finite set $ S $ of places of $ K $
such that
\begin{enumerate}
	
	\item for every place $ v \notin S $ and for all
	non-zero prime ideals $ \lambda $ coprime with
	$ \i^\ast v $
	the representation $ V_\lambda( M ) $ is unramified
	at $ v $;
	
	\item for these $ \lambda $ and $ v $ the
	characteristic polynomial of Frobenius at $ v $
	has coefficients in $ k[t] $ and is independent of $ \lambda $.
\qed
\end{enumerate}
\end{proposition}

In other words: the $ V_\lambda( M ) $ form a ``strictly compatible system'' of Galois representations \`{a} la Serre \cite{Serre68}. \index{compatible system of Galois representations}

\begin{example}
Assume that $\i $ is injective and that $ K $ is finite over $ k(\th) $.
Let $ C $ be the \emph{Carlitz} $ t $-motif over $ K $. This is the rank
one effective $ t $\dash motif given by
\[
	C = \big( K[ t ] e,\, \tau e \mapsto ( t - \th ) e \big).
\]
Let $ v $ be a finite place of $ K $ (\emph{i.e.} $ v $ does not lie
above the place $ \th = \infty $ of $ k( \th ) \subset K $.) 
Let $ f \in k[ \theta ] $ be
a monic generator of the ideal in $ k[ \theta ] $ corresponding to
the norm of $ v $ in $ k( \theta ) \subset K $. One verifies 
that
\begin{enumerate}
	\item the representation $ V_\lambda( C ) $ is unramified at
		$ v $ for all $ \lambda $ coprime with $ \i^\ast v $;
	\item for such $ \lambda $ we have that $ \Frob_v $ acts
		as $ f( t )^{ -1 } \in k( t ) $.
\end{enumerate}
So $ C $ plays the role of the Lefschetz motif $ \ZZ( -1 ) $.
\end{example}

\subsection{Abelian $ t $-modules}\label{tmodules}

Let $ ( M, \sigma ) $ be a $ \sigma $-module over $ K $.
Now consider the functor
\[
	E_M \colon
	\{  K \text{-algebras} \} \to 
	\{  k[t] \text{-modules} \} \colon
	R \mapsto \Hom_{ K[\sigma] }( M, R ),
\]
where $ R $ is a left $ K[\sigma] $\dash module through $\sigma r = r^q $.
This functor is representable by an affine $ k[t] $-module scheme over $ K $, which
is not necessarily of finite type over $ K $.

Conversely, given a $ k[t] $-module scheme $ E $ over $ K $ define 
\[
	M_E := \Hom_{\text{gr.sch.}/K}( E, \GG_a ),
\]
which is naturally a left $ K[ t ] $-module.  The $q$\dash th power Frobenius 
map $ \tau : \tau^\ast\GG_a \to \GG_a $ induces a linear map $ \sigma: M'_E \to M_E $ which
makes $ M_E $ into a $ \sigma $\dash module over $ K[t] $.

\begin{proposition}[\S 1 of \cite{Anderson86}, \S 10 of \cite{Stalder07}]\label{modmot}
The functors $ M \mapsto E_M $ and $ E \mapsto M_E $ form a pair of
quasi-inverse anti-equivalences between the categories of 
effective $ t $\dash motifs $ M $ over
$ K $ that are finitely generated as left $ K[ \sigma ] $-modules
and the category of $ k[t] $-module schemes $ E $ over $ K $ that satisfy
\begin{enumerate}
	\item for some $ d \geq 0 $ the group schemes
	$ E_{\bar{K}} $ and $ \GG_{a,\bar{K}}^d $ are isomorphic;
	\item $ t - \th $ acts nilpotently on $ \Lie( E ) $;
	\item $ M_E $ is finitely generated as a $ K[ t ] $-module.
\end{enumerate}
These anti-equivalences commute with base change $ K \to K' $.\qed
\end{proposition}

\begin{definition}
A $ k[t] $-module scheme $ E $ satisfying the above three conditions
is called an \emph{abelian $ t $-module} of dimension $ d $. An 
abelian $ t $-module of dimension one is called a \emph{Drinfeld module}.
\end{definition}

The tangent space at the identity of $ E $ can be expressed
in terms of $ M_E $ as follows: \index{tangent space at zero}
\begin{proposition}[\emph{see} \cite{Anderson86}]\label{LieE}
$ \Lie_E( K ) = \Hom_K( M_E/\sigma M'_E, K ) $ as $ K[[ t-\th ]] $\dash modules.
\end{proposition}

Also the Galois representations associated with $ M_E $
can be expressed in terms of $ E $. If $ \lambda = ( f ) \subset k[t] $
is a non-zero prime ideal then define the $ \lambda $-adic
Tate module of $ E $ to be
\[
	V_\lambda( E ) :=
	( \varprojlim_n E[ f^n ]( K^\sep ) )
	\otimes_{ k[t]_\lambda } k(t)_\lambda.
\]
If $ M $ is the effective $ t $-motif associated with $ E $ then we have
\begin{proposition}
$ V_\lambda( M ) = \Hom ( V_\lambda( E ), k(t)_\lambda ) $. \qed
\end{proposition}

\subsection{Weights and Dieudonn\'{e} $t$-modules}

In this section we recall the definition of the weights of an effective $t$\dash motif 
\cite{Taelman09a}.
This will allow for a more useful equivalent description of the ``finitely generated over $ K[\sigma] $'' condition.

 As usual $k$ is a finite field of $q$ elements and
 $K$ a field containing $k$. Denote by $\tau$ the continuous
 endomorphism of the field of Laurent series $K\llp t^{-1}\rrp$ that
 fixes $t^{-1}$ and that restricts to the $q$\dash th power 
 map on $K$. Note that $ K\llp t^{-1}\rrp^\tau = k\llp t^{-1}\rrp $. If $ V $
 is a $K\llp t^{-1}\rrp$\dash vector space then we denote the pull-back
 by $ \tau $ of $ V $ by $ V' $.
 \begin{definition} \index{Dieudonn\'{e} $t$-module}
  A \emph{Dieudonn\'{e} $t$\dash module}
  over $K$ is a pair $(V,\sigma)$ consisting of
  \begin{enumerate}
   \item a finite-dimensional $K\llp t^{-1}\rrp$\dash vector space $V$
         and
   \item a $ K\llp t^{-1}\rrp$\dash linear isomorphism $\sigma:V'\to V$.
  \end{enumerate}
  A morphism of Dieudonn\'{e} $t$\dash module is
  a  $K\llp t^{-1}\rrp$\dash linear map compatible with $\sigma$.
 \end{definition}
 
 Just as with $\sigma$\dash modules one can also consider Dieudonn\'{e} $t$\dash module 
 as modules over the skew polynomial ring $ K\llp t^{-1}\rrp [\sigma] $ whose elements are
 polynomials in $\sigma$ and where $\sigma t = t \sigma $ and $ \sigma x = x^q \sigma $.

 Dieudonn\'{e} $t$\dash modules over a separably closed field admit a simple
 classification. The main `building blocks' are the
 following modules:
 \begin{definition}
  Let $\mu=s/r$ be a rational number with
  $(r,s)=1$ and $r>0$. The 
  Dieudonn\'{e} $t$\dash module $V_\mu$ is defined to be the
  pair $(V_\mu,\sigma)$ with 
  \begin{enumerate}
   \item $V_\mu := K\llp t^{-1}\rrp e_1\oplus\ldots\oplus
          K\llp t^{-1}\rrp e_r$
   \item $\sigma \tau e_i := e_{i+1}$ ($i<r$) and
	  $\sigma \tau e_r := t^se_1$
  \end{enumerate}
 \end{definition}
 
 \begin{proposition}\label{dieuclass}
 If $ K $ is separably closed then the category of Dieudonn\'{e} $t$\dash modules over
 $ K $ is semi-simple. The simple objects are
  the $ V_\mu $ with $ \mu \in \QQ $ and $V_\mu\cong V_\nu$ if and only if $ \mu = \nu $.
 \end{proposition}

 Note that this classification is formally identical to
 the classification of the classical ($p$\dash adic) Dieudonn\'{e}
 modules \cite{Dieudonne57}.

 \begin{proof}
  This is shown in \cite[Appendix B]{Laumon96}. Although
  the statements therein are made only for a particular field
  $K$, nowhere do the proofs make use of anything stronger
  then the separably closedness of $K$.
 \end{proof}

 Let $ V / K $ be a Dieudonn\'{e} $t$\dash module. Then
 by the above Proposition there exist rational numbers
 $\mu_1,\ldots,\mu_n$ such that
 \[
 	V_{K^\sep} \cong  V_{\mu_1} \oplus\cdots\oplus
                         V_{\mu_n},
 \]
 where $ V_{K^\sep} $ denotes the completed tensor product
 $ V \hat{\otimes}_K K^\sep $. 
 We call these rational numbers the \emph{weights} of $V$. 
 If $ M/K $ is a $ \sigma $\dash module which is finitely generated over $ K[t] $ then 
 \[
 	M\llp t^{-1}\rrp := M \otimes_{K[t]} K\llp t^{-1}\rrp
 \]
 is naturally a Dieudonn\'{e} $t$\dash module and we define
 the weights of $ M $ to be the weights of $ M\llp t^{-1}\rrp $. \index{weight}
  
\begin{proposition}\label{fingen}
 Let $ M $ be a $ \sigma $\dash module, free and finitely generated
 as $ K[t] $\dash module.
 Then $ M $ is finitely generated as $ K[\sigma] $\dash module if and only if
 $ M $ is a $\sigma$-module and all weights of $ M $ are positive.
\end{proposition}

\begin{proof}
The ``if'' part is shown in Theorem 5.3.1 of \cite{Taelman09a}, the ``only if'' part in Proposition 8 of \cite{Taelman09b}.
\end{proof}

\begin{corollary}
The functors $ M \mapsto E_M $ and $ E \mapsto M_E $ form a pair of
quasi-inverse anti-equivalences between the categories of 
effective $ t $\dash motifs $ M / K $ whose weights are positive and the category 
of abelian $ t $-modules over $ K $.
\end{corollary}

Finally we prove a useful fact about extensions of Dieudonn\'{e} $t$-modules over arbitrary fields. 

\begin{proposition}\label{dieuext}
Let $ V  $ and $ W $ be Dieudonn\'{e} $ t $-modules over $ K $. If $ V $ and $ W $ have no weights in common then $ \Hom( V, W ) = \Ext^1( V, W ) = 0 $ in the category of Dieudonn\'{e} $t$-modules over $ K $.
\end{proposition}

\begin{proof}
A non-zero morphism $ V \to W $ would induce a non-zero morphism
$ V_{K^\sep} \to W_{K^\sep} $ after base change, and such a morphism cannot exist
by Proposition \ref{dieuclass}.

If $ U $ is an extension of $ V $ by $ W $ then $ U_{K^\sep} $ splits by Proposition \ref{dieuclass}. So $ U $ is a form over $ K $ of the split extension
$ V_{K^\sep} \oplus W_{K^\sep} $ over $ K $. Such forms are classified by the Galois
cohomology group 
\[
	 {\mathrm H}^1\!\big( \Gal( K^\sep/K ), \Aut( V_{K^\sep} \oplus W_{K^\sep} ) \big)
\]
but since the weights of $ V $ and $ W $ are disjoint we have
\[
	\Aut( V_{K^\sep} \oplus W_{K^\sep} ) = \Aut( V_{K^\sep} ) \times \Aut( W_{K^\sep} ),
\]
and it follows that any extension of $ V $ by $ W $ splits over $ K $.
\end{proof}

\section{Algebraic theory of $1$-$t$-motifs}\label{algebraictheory}

\begin{theorem}[repeated from \S \ref{statements}]\label{thmext}
Let $ M $ be a left $ K[t,\sigma] $\dash module which is free and finitely
generated both as a $ K[t] $\dash module and as a $ K[\sigma] $\dash module. There is an
isomorphism of $ k[t] $-modules
\[
	 \Hom_{K[\sigma]}( M, K ) \longrightarrow \Ext^1_{K[t,\sigma]}( M, K[t] ),
\]
functorially in $ M $ and compatible with field extensions $ K \to K' $.
\end{theorem}

\begin{corollary}
If $ M $ is the effective $ t $-motif associated to an abelian $t$-module $ E $ 
we have a natural isomorphism of $ k[t] $-modules
$ E( K ) = \Ext^1_{K[t,\sigma]}( M, K[t] ) $.
\end{corollary}

\begin{proof}[Proof of Theorem \ref{thmext}]
We will produce the desired isomorphism as the composition
of two isomorphisms
\[
	\Hom_{K[\sigma]}( M, K )
	\longleftarrow
	\Hom_{ K[t,\sigma] }( M, K\llp t^{-1}\rrp/K[t] )
	\longrightarrow
	\Ext^1_{ K[t,\sigma] } ( M, K[t] ).
\]
	
The residue map (coefficient of $ t^{-1} $) defines a
natural isomorphism
\[
	\Hom_{K[\sigma]}( M, K )
	\stackrel{\Res}{\longleftarrow}
	\Hom_{ K[t,\sigma] }( M, K\llp t^{-1}\rrp/K[t] ),
\]
which is the first of the two isomorphisms.

To obtain the second isomorphism apply
 $ \Hom_{ K[t,\sigma] }( M, - ) $ to the short exact sequence
\[
	0 \to K[t] \to K((t^{-1})) \to K((t^{-1}))/K[t] \to 0,
\]
which yields a connecting homomorphism
\[
	\Hom_{ K[t,\sigma] }( M, K\llp t^{-1}\rrp/K[t] )
	\longrightarrow
	\Ext^1_{ K[t,\sigma] }( M, K[t] ).
\]
To show that it is an isomorphism it suffices to prove that
\[
	\Hom_{ K[t,\sigma] }( M, K\llp t^{-1}\rrp ) =
	\Ext^1_{ K[t,\sigma] }( M, K\llp t^{-1}\rrp ) = 0.
\]

Consider the $ k[t] $-linear map
\[
	\delta\colon \Hom_{K[t]}( M, K\llp t^{-1}\rrp )
	\to \Hom_{K[t]}( M',  K\llp t^{-1}\rrp )
	\colon 	f \mapsto f\circ \sigma - \tau \circ f.
\]
The kernel of $ \delta $ is $ \Hom_{ K[t,\sigma] }( M, K\llp t^{-1}\rrp ) $ and the cokernel
is $ \Ext^1_{ K[t,\sigma] }( M, K\llp t^{-1}\rrp ) $. Moreover, $ \delta $ 
coincides with the $ k\llp t^{-1}\rrp $-linear map
\begin{eqnarray*}
	\Hom_{K\llp t^{-1}\rrp}( M\llp t^{-1}\rrp, K\llp t^{-1}\rrp )
	&\longrightarrow&
	\Hom_{K\llp t^{-1}\rrp}( M'\llp t^{-1}\rrp,  K\llp t^{-1}\rrp ) \\
	f &\longmapsto& f\circ \sigma - \tau \circ f
\end{eqnarray*}
and its kernel and cokernel are  $ \Hom( M\llp t^{-1}\rrp, K\llp t^{-1}\rrp ) $ and
$ \Ext( M\llp t^{-1}\rrp, K\llp t^{-1}\rrp ) $  respectively (both in the category of Dieudonn\'{e} $t$-modules), which
vanish by Proposition \ref{dieuext}.
\end{proof}

\begin{remark}
Let $ \mu : M \to K $ be an element of
$ E_M(K) = \Hom_{K[\sigma]}( M, K )  $. Then the corresponding
extension $ 0 \to K[t] \to M_\mu \to M \to 0 $ can
be described explicitly as
\begin{equation}\label{explicitext}
	M_\mu :=
	\left\{ ( m, f ) \in M \times K\llp t^{-1}\rrp 
		:
		f - \sum_{i=0}^\infty \mu( t^i m ) t^{-i-1}
		\in K[t]
	\right\}.
\end{equation}
In case $ M $ is a tensor power of the Carlitz $t$-motif such a description is used in
\cite{Anderson90}.
\end{remark}

\begin{definition}
A $1$-$t$-module over $ K $ is a triple $ ( X, E, u ) $ consisting of
\begin{enumerate}
\item a free and finitely generated $ k[ t ] $-module
	$ X $ equipped with a continuous action of
	$ \Gal(K^\sep/K) $;
\item an abelian $ t $-module $ E $ over $ K $;
\item a $ \Gal(K^\sep/K) $-equivariant homomorphism of $ k[ t ] $-modules
	$ u : X \to E( K^\sep ) $.
\end{enumerate}
\end{definition}

We now construct a $ t $-motif $\tilde{M}$ associated with a
$ 1 $-$ t $-module $ ( X, E, u ) $. From $ X $
one builds the $ t $-motif
\[
	V := \Hom_{ k[t] } ( X, K^\sep[ t ] )^{\Gal(K^\sep/K)}.
\]
This is a $ K[ t ] $-module which is free of
rank the $ k[ t ] $-rank of $ M $ and since 
$ \tau:  K^\sep[ t ] \to K^\sep[ t ] $ commutes with
the action of $ G_K $ it defines a map
$ \sigma: V \to V $ which makes $ V $
into a $ t $-motif over $ K $.  This $ V $ is an Artin $t$-motif:

\begin{definition}
An \emph{Artin $t$\dash motif over K} is an effective $t$\dash motif $ V $ over $ K $
such that $ V_{K^\sep} $ is isomorphic to the effective $t$-motif $ K^\sep[t]^r $,
where $r$ is the rank of $ V $.
\end{definition}

One easily verifies that $ X \mapsto V $ is an anti-equivalence from the category of
continuous $ \Gal(K^\sep/K) $\dash modules over $ k[t] $ that are free of finite rank over $ k[t] $ to the category of Artin $t$\dash motifs over $ K $.

Let $ M $ be the effective $t$-motif associated to $ E $.
The map $ u $ defines through Theorem \ref{thmext} an
extension of $ K[ t, \sigma ] $-modules
\[
	0 \to V \to \tilde{M} \to M \to 0
\]
which defines the $ t $-motif $ \tilde{M} $ associated
with $ ( X, E, u ) $.

\begin{definition}
A \emph{1-$t$-motif} is a $ t $-motif $ \tilde{M} $
which fits into an exact sequence
\[
	0 \to V \to \tilde{M} \to M \to 0
\]
with $ V $ an Artin $ t $-motif and $ M $ a 
$ t $-motif of strictly positive weights.
\end{definition}

Note that the exact sequence is uniquely determined by
$ \tilde{M} $, so from propositions \ref{modmot}, \ref{fingen}, and
\ref{thmext} we obtain:

\begin{corollary}
The above construction $ ( X, E, u ) \mapsto \tilde{M} $
defines an anti-equivalence from the category of
$ 1 $-$ t $-modules over $ K $ to the category of
$ 1 $-$ t $-motifs over $ K $. \qed
\end{corollary}

\smallskip

Finally, we briefly explain the relationship between Theorem
\ref{thmext} and the corresponding result of Papanikolas and Ramachandran
\cite{Papanikolas03}. Let $ E $ be a Drinfeld module of rank $ r $ over $ K $ and $ M=M_E$ 
the corresponding effective $t$\dash motif. Define $ M^\vee $ by
\[
	M^\vee := \Hom_{K[t]}( M, C ).
\]
Note that $ (M^\vee)' = \Hom_{K[t]}( M', C' ) $.
\begin{definition}\label{Mvee}
 We make $ M^\vee $ into an effective $t$\dash motif by the rule
\[
	\sigma_{M^\vee}\colon \Hom_{K[t]}( M', C' ) \to \Hom_{K[t]}( M, C ) \colon
		f \mapsto \sigma_C \circ f \circ \sigma_M^{-1}.
\]
\end{definition}

The map $ \sigma_M $ is not surjective, but the above formula should be read as follows. 
After extending scalars from $ K[t] $ to $ K(t) $, the linear map $ \sigma_M $ becomes an
isomorphism and one has that
$ \sigma_C \circ f \sigma_M^{-1} $ maps
\[
	 (M^\vee)' \subset (M^\vee)' \otimes_{K[t]} K(t)
\]
into
\[
	 M^\vee \subset M^\vee \otimes_{K[t]} K(t).
\]

\begin{proposition}\label{PR}
There is a natural short exact sequence
\[
	0 \longrightarrow \Ext^1( M, K[t] )
	\longrightarrow
	\Ext^1( C, M^\vee ) \longrightarrow
	\Hom_{K[t]}\!\left( C', \frac{M^\vee}{\sigma_{M^\vee}(M^\vee)'} \right)
	\longrightarrow 0.
	\qedhere 
\]
of $ k[t] $\dash modules.
\end{proposition}

 This recovers the short exact sequence of Theorem 1.1 (b) of \cite{Papanikolas03}. To see that
\[
	\Hom_{K[t]}\!\left( C', \frac{M^\vee}{\sigma_{M^\vee}(M^\vee)'} \right)
\]
is indeed an  $ (r-1) $-dimensional $K$-vector space, first assume that $r>1$. Note that $C$ is free of rank $1$ over $K[t]$ and that the quotient
$ M^\vee / \sigma_{M^\vee}(M^\vee)' $ is the dual of the tangent space of
the  $t$-module corresponding to $M^\vee$. Since $M^\vee$ is pure of weight $ 1-1/r $ and has
rank $ r $, it has dimension $r-1$. If $r=1$ then $ M^\vee $ does not correspond to an abelian
$t$-module, but it is easy to see directly that $ M^\vee = \sigma_{M^\vee}(M^\vee)' $.

\begin{proof}[Proof of the proposition]
The module $ \Ext^1( M, K[t] ) $ is the cokernel of the map
\[
	\delta_1\colon \Hom_{K[t]}( M, K[t] )
	\to \Hom_{K[t]}( M',  K[t] )
	\colon 	f \mapsto f\circ \sigma - \tau \circ f.
\]
This map is naturally isomorphic with the map
\[
	\delta_2\colon \Hom_{K[t]}( C, M^\vee ) \to
		\Hom_{K[t]}( C', (M^\vee)' )
\]
given by
\[
	f \mapsto \sigma_{M^\vee}^{-1} \circ f \circ \sigma_C - f'.
\]
(The double composition is well-defined in the sense of  the remark after Definition \ref{Mvee}.) Consider the square
\begin{equation*}
\begin{CD}
	\Hom_{K[t]}( C, M^\vee ) @>{\delta_2}>>
	\Hom_{K[t]}( C', (M^\vee)' )  \\
	@V{\id}VV @VV{\sigma_{M^\vee}}V \\
	\Hom_{K[t]}( C, M^\vee ) @>{\delta_3}>>
	\Hom_{K[t]}( C', M^\vee )  \\
\end{CD}
\end{equation*}
where the bottom map is the map
\[
	 \delta_3\colon f \mapsto f\circ \sigma_C - \sigma_{M^\vee} \circ f'.
\]
Note that the square commutes and that the cokernel of the bottom map is
$ \Ext^1( C, M^\vee ) $. The right-hand side is injective with cokernel
\[
	\Hom_{K[t]}( C', \frac{M^\vee}{\sigma_{M^\vee}(M^\vee)'} ),
\]
which yields the desired short exact sequence.
\end{proof}


\section{Uniformization and Hodge structures}\label{unifandhodge}

For the remainder of this paper we assume that $ K $ is a local field and that
$ | \theta | > 1 $, so in particular
$ \i : k[ t ] \to K : t \mapsto \theta $ is injective. We also
fix an algebraic
closure $ \bar{K} $ of $ K $ and a completion $ C $ of $ \bar{K} $. The
field $ C $ is algebraically closed.

Although we will consider effective $t$-motifs over $ C $, we
will need to \emph{assume} that they are defined over $ \bar{K} \subset C $.
This is because we will use results of Anderson \cite{Anderson86} that use a
locally compact field of definition. It is possible that these results could be generalized to include all effective $t$-motifs over $ C $.

\subsection{Uniformization of abelian $t$-modules}

\begin{proposition}[\emph{see} \S 2 of \cite{Anderson86}]
Let $ E $ be an abelian $ t $-module over $ \bar{K} $. \index{exponential map}
\begin{enumerate}
\item There exists a unique entire $ k[t] $-module homomorphism
$ \exp_E : \Lie_E( C ) \to E( C ) $ that
is tangent to the identity map;
\item The kernel of $ \exp_E $ is a finitely
generated free discrete sub-$ k[t] $-module in $ \Lie_E( C ) $.
\end{enumerate}
\end{proposition}

When $ \exp_E $ is surjective we say that $ E $ is \emph{uniformizable}, 
\index{uniformizability} and in
that case we have a short exact sequence of $ k[t] $-modules
\begin{equation}\label{unifseq}
	0
	\longrightarrow
	\Lambda_E
	\longrightarrow
	\Lie_E( C )
	\longrightarrow
	E( C )
	\longrightarrow
	0
\end{equation}
where $ \Lambda_E := \ker( \exp_E ) $.

\begin{example} 
Drinfeld modules are uniformizable \cite{Drinfeld74E}.
\end{example}

\begin{remark}
There exist abelian $t$-modules that are not uniformizable \cite{Anderson86}.
\end{remark}

Whether or not an abelian $t$-module $ E $ is uniformizable may be read off from the
associated $ t $-motif $ M $. To do so, we need the Tate algebra:
\[
	C\{ t \} := 
	\left\{ \sum f_i t^i \in C\llb t \rrb : | f_i | \to 0\text{ as } i \to \infty 
	\right\}.
\] 
This algebra has the following properties:
\begin{enumerate}
	\item $ C\{ t \} $ is a $ C[ t ] $-algebra;
	\item $ \tau : \sum f_i t^i \mapsto \sum f_i^q t^i $ defines an endomorphism
		of the $ k $-algebra $ C\{ t \} $;
	\item $ C\{ t \}^{\tau} = k[ t ] $.
\end{enumerate}

Given a  $ \bar{K}[t] $-module $ M $ we denote by $ M\{t\} $ the tensor product
$ M \otimes_{ \bar{K}[t]} C\{ t \} $. If $ M $ is a $\sigma$\dash module then
$\sigma$ extends to a $ C\{ t \} $\dash linear map $ M'\{t\} \to M\{t\} $. Also
the canonical map $ \tau $ extends to a $ k[t] $\dash linear map $ \tau: M\{t\} \to M'\{t\} $.
The invariants $ M\{t\}^{\sigma\tau} $ form a $k[t]$\dash module.

\begin{definition}[{\cite[\S 2]{Anderson86}}]
  An effective $t$-motif $M$ over $ \bar{K} $ is said to be \emph{analytically trivial}
  if the natural map $ M\{t\}^{\sigma\tau} \otimes_{k[t]} C\{t\} \to M\{t\} $ is
  an isomorphism. \index{analytically trivial}
\end{definition} 

\begin{proposition}[{\cite[\S 2]{Anderson86}}]
Let $ E / \bar{K} $ be an abelian $ t $-module and let $ M $ be the associated effective
$ t $-motif. Let $ r $ be the rank of $ M $. The following are equivalent:
\begin{enumerate}
	\item $ M\{ t \}^{\sigma\tau} $ is free of rank $ r $ as $ k[t] $-module;
	\item $ M $ is analytically trivial;
	\item $ E $ is uniformizable.\qed
\end{enumerate}
\end{proposition}

\begin{proposition}\label{antrivext}
If $ M_1 $ and $ M_2 $ are analytically trivial $t$-motifs then any extension
of $ M_1 $ by $ M_2 $ is analytically trivial.
\end{proposition}

\begin{proof}
It suffices to show that any extension of the ``trivial'' $C\{t\}[\sigma]$\dash module
$C\{t\}$ by itself splits. Let $ D $ be such an extension. Then $ D $ has a basis
$(e_1, e_2)$ such that $\tau e_1 = e_1$ and $\tau e_2 = e_2 + fe_1 $ for some $f\in C\{t\}$.

The extension splits if and only if there exists a $g \in C\{t\} $ such that 
\[
	f = \tau(g) - g.
\]
Note that if $a$ is an element of $ C $ with $|a|<1$ then there is a (unique) $ x\in C $
such that $ a = x^q - x $ and  $ |x|=|a| $. 
If we write $f$ as $ f=\sum_i f_it^i $ then we can take for $ g $ any power series
$ \sum_i g_it^i $ with $ g_i^q - g_i = f_i $ for all $i$, as long as we choose $g_i$ such
that $ |g_i| = |f_i| $ for all $ i $ sufficiently large.
\end{proof}

Next we will define Hodge structures of analytically trivial $ t $-motifs and use them to recover the uniformization sequence (\ref{unifseq}) from $ M $.

\subsection{Hodge structures of $t$-motifs}

\begin{definition}
A (function field) \emph{Hodge structure} is a diagram $ H_1 \to H_2 $ consisting of
\begin{enumerate}
	\item a $ k[t] $-module $ H_1 $;
	\item a $ C\llb t-\th \rrb $-module $ H_2 $;
	\item a $ k[t] $-linear map $ H_1 \to H_2 $.
\end{enumerate}
A morphism of Hodge structures is a commutative square. The category of Hodge structures
is denoted by $ \Ho $. \index{Hodge structure}
\end{definition}

The category $\Ho$ is an abelian category.

\begin{remark}
These structures were introduced by Pink in \cite{Pink97}. His definition contains both more data and more conditions, but for the purposes of this note the above simple definition suffices. It would be more correct to call our triples pre-Hodge structures.
\end{remark}

If $ M/ C $ is a uniformizable effective $ t$\dash motif then every element
of $ M\{t\}^{\sigma\tau} $
has an infinite radius of convergence \cite[3.13]{Anderson04}. In particular, it makes sense to talk about the
power series expansion of such an element around $ t=\th $. This gives the
Hodge structure $ H_M $ associated with $ M $, namely
\[
	H_M := \big[\, M\{t\}^{\sigma\tau} \longrightarrow M\llb t-\th \rrb\, \big],
\]
where $ M\llb t-\th\rrb := M \otimes_{ C[t]} C\llb t-\th \rrb $. If $ M $ is a $t$\dash motif
over a subfield of $ C $ then we define $ H_M $ to be Hodge structure of the extension of scalars
of $ M $ to $ C $.

In the remainder of this section we relate the Hodge structure  $ H = H_M $ of an
effective $ t $-motif $ M $ with strictly positive weights with the uniformization of the associated abelian $ t $-module $ E= E_M $. We largely follow Anderson \cite{Anderson86}, and for most of the part are merely rephrasing him using the language of Hodge structures.

A special role will be played by the following Hodge structures:
\[
	\One := \big[\, k[t] \to C\llb t-\th\rrb\, \big],
\]
\[	
	A    := \big[\, k[t] \to C\llp t-\th\rrp\, \big],
\]
and
\[		
	B    := \big[\, 0 \to C\llp t-\th\rrp/C\llb t-\th\rrb\, \big].
\]
Note that they naturally sit in a short exact sequence $ 0 \to \One \to A \to B \to 0 $.

\begin{proposition}\label{twohodges}
For all uniformizable abelian $t$-modules $ E $, and functorially in $ E $, there is a commutative square of $ k[t] $-modules 
\begin{equation*}
\begin{CD}
	\Lambda_E @>{\sim}>> \Hom_\Ho\!\big( H, A \big) \\
	@VVV @VVV \\
	\Lie_E( C ) @>{\sim}>> \Hom_\Ho\!\big( H, B \big),
\end{CD}
\end{equation*}
where $ H $ is the Hodge structure of $ M_E $. Both horizontal maps are isomorphisms and the right map is induced by the natural map $ A \to B $.
\end{proposition}

Before giving the proof we first state the following consequence:

\begin{corollary} $ \Lambda_E $ generates $ \Lie_E(C) $ as a $ C\llb t-\th \rrb $-module.
\end{corollary}

This corollary generalises Corollary 3.3.6 of \cite{Anderson86},
which employs the additional hypothesis that $ E $ be pure.

\begin{proof}[Proof of the corollary]
By Proposition \ref{twohodges} it suffices to check that the map
\[
	\Hom_\Ho( H, A )\otimes_{k[t]} C\llb t-\th\rrb \to \Hom_\Ho( H, B ) 
\]
induced by $ A \to B $ is surjective, which is straightforward.
\end{proof}

The proof of Proposition \ref{twohodges} fills the rest of this section.

\begin{lemma}\label{lemmaone}
The sequence of $ C\llb t-\th \rrb $-modules
\begin{equation}\label{seq1}
	0 \longrightarrow
	M\{t\}^{\sigma\tau} \otimes_{k[t]} C\llb t-\th\rrb \longrightarrow
	M[[t-\th]] \overset{-\epsilon}
	{\longrightarrow}
	M/\sigma M' \longrightarrow
	0
\end{equation}
is exact.
\end{lemma}

\begin{proof}
Consider the commutative square of $ C[[ t-\th ]] $\dash modules
\begin{equation*}
\begin{CD}
	M'\{t\}^{\tau\sigma} \otimes_{k[t]} C\llb t-\th \rrb @>>>
	M'\llb t-\th \rrb \\
		@V{\sigma}VV @V{\sigma}VV \\
	M\{t\}^{\sigma\tau} \otimes_{k[t]} C\llb t-\th\rrb @>>>
	M\llb t-\th \rrb. \\
\end{CD}
\end{equation*}
We first show that the top map is an isomorphism. To do so, it suffices to verify 
that its determinant is an isomorphism. Since the formation of the this map commutes
with tensor products and exterior powers, we can reduce to the case where $ M $ is
the Carlitz $t$\dash motif. In that case a simple computation shows that $ M'\{t\}^{\tau\sigma} $
is generated by
\[
	(-\th)^{\frac{-q}{q-1}} \prod_{i>0} \left( 1 - \frac{ t }{ \th^{q^i} } \right)e
\]
whose image is a generator of $ M'\llb t-\th \rrb $.

Now since the left and upper morphisms are isomorphisms, from which we obtain an isomorphism between the cokernels of the right and lower maps. The cokernel of the right map is isomorphic with
$ M / \sigma M' $, so combining these we get the desired short exact sequence.
\end{proof}

Dualizing the sequence (\ref{seq1}) \`{a} la \S \ref{torsionduality} gives the bottom isomorphism
\[
	\Hom_\Ho( H , B ) \overset{\sim}{\longleftarrow} \Hom_C( M_C/\sigma M_C', C ) = 
	\Lie_E(C)
\]
in the commutative square of Proposition \ref{twohodges}.

\begin{lemma}\label{andres}
Let $ N $ be a non-negative integer. For all $ m \in M $ and for all
$ \epsilon \in \Lie_E(C) = \Hom_C(  M_C/\sigma M'_C, C )  $ 
the Laurent series
\begin{equation}\label{series}
	\beta_N( m, \epsilon ) : = \sum_{i > -N}  \exp( t^{-i-1} \epsilon )( m ) t^i
	\in C\llp t\rrp
\end{equation}
defines a function which is meromorphic in $ t=\th $ and
\begin{equation}\label{resbeta}
	\Res_{t=\th} \beta_N( m, \epsilon ) = - \epsilon( m + \sigma M' ).
\end{equation}
\end{lemma}

\begin{proof}
The validity of the statements does not depend on $ N $. For $ N = 0 $ a proof is in \cite[3.3.4]{Anderson86}.
\end{proof}

Consider the map
\[
	\Lie_E( C ) \times M \to C\llp t \rrp  :
	( \epsilon, m ) \mapsto \beta_N( m, \epsilon ) = 
	\sum_{i \geq -N} \exp_E( t^{-i-1}\epsilon )( m ) t^i,
\]
which is $ k[t] $-linear in the first and $ C[\sigma] $-linear in the second argument. The restriction
\[
	\Lambda_E \times M \to C\llp t \rrp
\]
is independent of $ N $. Also, it is $ C[ t, \sigma ] $-linear in $ M $ so it extends
to a  map
\[
	\Lambda_E \times M\{t\} \to C\llp t \rrp,
\]
which is $ C\{ t \}[ \sigma ] $-linear in its second argument.
Restriction to $ \sigma $-invariants now gives a
$ k[ t ] $-bilinear form
\[
	\Lambda_E \times M\{t\}^{\sigma\tau} \to k\llp t \rrp.
\]

\begin{lemma}\label{pairing}
The above bilinear form takes values in $ k[ t ] $ and the
resulting 
\[
	\Lambda_E \times M\{t\}^{\sigma\tau} \to k[t]
\]
is a perfect pairing.
\end{lemma}

\begin{proof}
See \P 2.6 of \cite{Anderson86}.
\end{proof}

Lemma \ref{pairing} provides the top isomorphism
\[
	\Lambda_E \overset{\sim}{\longrightarrow} 
	\Hom_{k[t]}( M\{t\}^{\sigma\tau}, k[t] ) = \Hom_\Ho( H, A ),
\]
of the commutative square whose existence is claimed in Proposition \ref{twohodges}
and Lemma \ref{andres} shows commutativity. This finishes the proof of Proposition
\ref{twohodges}.

\section{Transcendental theory of $1$-$t$-motifs}\label{transtheory}

As in the previous section, $ M $ denotes a uniformizable $ t $-motif over $  \bar{K} $ 
whose weights are strictly positive and $ E = E( M ) $
be the corresponding abelian $ t $-module. The Hodge structure of $ M $ is
denoted by $ H $. In this section we will show that any extension $ \tilde{H} $ of
$ H $ by $ \One $ comes from a unique $1$-$t$-motif $ \tilde{M} $, extension of $ M $ by $ \One $.

\index{extension of Hodge structures}

Applying $ \Hom( H, - ) $ to the short exact sequence 
\[
	0 \to \One \to A \to B \to 0
\]
of Hodge structures we obtain a long exact sequence of $ k[t] $-modules
\[
	\cdots \to \Hom( H, A ) \to \Hom( H, B ) \to \Ext^1( H, \One ) \to \cdots
\]

\begin{lemma}\label{exthodgecal}
The sequence
\[
	0 \to \Hom( H, A ) \to \Hom( H, B ) \to \Ext^1( H, \One ) \to 0
\]
is exact.
\end{lemma}

\begin{proof}
Injectivity of $ \Hom( H, A ) \to \Hom( H, B ) $ follows from Proposition \ref{twohodges}. The only thing that is new is the surjectivity of $ \Hom( H, B ) \to \Ext^1( H, \One ) $. It suffices
to show $ \Ext^1( H, A ) = 0 $. So let
\[
	0 \to A \to \tilde{H} \to H \to 0
\]
with
\[
	 \tilde{H} = \big[ \tilde{H}_1 \to \tilde{H}_2 \big]
\]
be an extension of Hodge structures. Choose a splitting
$ s_1 : \tilde{H_1} \to k[t] $ of
\[
	0 \to k[t] \to \tilde{H}_1 \to H_1 \to 0.
\]
Note that by Lemma \ref{lemmaone} the natural map
\[
	H_1 \otimes_{k[t]} C\llp t-\th \rrp
	\longrightarrow
	H_2 \otimes_{C\llb t-\th\rrb} C\llp t-\th \rrp
\]
is an isomorphism, and hence also 
\[
	\tilde{H}_1 \otimes_{k[t]} C\llp t-\th \rrp
	\longrightarrow
	\tilde{H}_2 \otimes_{C\llb t-\th\rrb} C\llp t-\th \rrp
\]
is an isomorphism. It follows that the splitting $ s_1 $ induces
a compatible splitting $ s_2 : \tilde{H}_2 \to C\llp t-\th \rrp $ and hence
that $ \tilde{H} $ is a split extension of $ H $ by  $ A $.
\end{proof}

Consider now the following commutative diagram:
\begin{equation*}
\begin{CD}
	0 @. 0 \\
	@VVV @VVV \\
	\Lambda_E @>{\sim}>> \Hom_{\Ho}( H, A ) \\
	@VVV @VVV \\
	\Lie_E( C ) @>{\sim}>> \Hom_{\Ho}( H, B ) \\
	@V{\exp_E}VV @VVV \\
	E( C ) @. \Ext^1( H, \One ) \\
	@VVV @VVV \\
	0 @. 0
\end{CD}
\end{equation*}
The left column is the uniformization short exact sequence. The right column is a short exact sequence from Proposition \ref{exthodgecal}. The horizontal isomorphisms are given by Proposition \ref{twohodges}.

\begin{theorem}[repeated from \S {\ref{statements}}]\label{extext}
The natural map 
\[ h : \Ext^1( M_C, C[t] ) \to \Ext^1( H, \One ) \]
is an isomorphism.
\end{theorem}

In fact we will prove the following stronger statement:

\begin{proposition}\label{pextext}
 The unique isomorphism
\[
	u : E( C ) \to \Ext^1( H, \One )
\]
that makes the above diagram commute coincides with the natural map
\[
	h : E( C ) = \Ext^1( M_C, C[t] ) \to \Ext^1( H, \One ).
\]
\end{proposition}

The proof of this proposition takes up the rest of this section. The idea of the proof is
to first prove that $ u $ and $ h $ coincide on the torsion submodule of $ E( C ) $, and
then to use a kind of density argument to conclude that they coincide on all of $ E( C ) $.

\begin{lemma}\label{torsion}
	For all integers $ N> 0 $ and  $ x \in E( C ) $ with $ t^Nx=0 $ we have $ h(x)=u(x) $.
\end{lemma}

\begin{proof}
We need to show that for all $ N>0 $ and for all $ \lambda \in \Lambda_E $ the image of
\[
	\epsilon := t^{-N} \lambda \in \Lie_E
\]
in $ \Ext^1( H, \One ) $ does not depend on the chosen path in the square
\begin{equation*}
\begin{CD}
	\Lie_E( C ) @>{\alpha}>> \Hom_{\Ho}( H, B ) \\
	@V{\exp_E}VV @V{\delta}VV \\
	\Ext^1( M, \One ) @>h>> \Ext^1_\Ho( H, \One ). 
\end{CD}
\end{equation*}
We first construct the $1$-$t$-motif $ \tilde{M} $ corresponding to $ \exp(\epsilon) $. 
By (\ref{explicitext}) it is defined by the short exact sequence
\[
	0 \longrightarrow
	\tilde{M}
	\longrightarrow
	M \oplus t^{-N}C[t]
	\overset{(\beta_{N}(\epsilon,-), -1)}{\longrightarrow}
	t^{-N}C[t] / C[t]
	\longrightarrow
	0.
\]
Applying the functor $ \Ho $ one finds that $ h( \exp(\epsilon) ) $ is given by the
kernel of the following morphism of Hodge structures (from the left column to the right column):
\begin{equation}\label{hexp}
\begin{CD}
	 M\{t\}^{\sigma\tau} \oplus t^{-N}k[t] @>>> t^{-N}k[t]/k[t]  \\
 	            @VVV                            @VVV \\
	 M\llb t-\th\rrb \oplus C\llb t-\th \rrb    @>>>     0,                     
\end{CD}
\end{equation}
where the upper map is given by
\[
	( m , f ) \mapsto \beta_{N}(\epsilon, m ) - f.
\]
The kernel of (\ref{hexp}) coincides with the kernel of
\begin{equation}\label{hexp2}
\begin{CD}
	 M\{t\}^{\sigma\tau} \oplus t^{-N}k[t] @>>> t^{-N}k[t]/k[t]  \\
 	            @VVV                            @VVV \\
	 M\llb t-\th \rrb \oplus C\llp t-\th\rrp    @>>>     C\llp t-\th \rrp/C\llb t-\th \rrb,                     
\end{CD}
\end{equation}
where both horizontal maps are given by $ ( m , f ) \mapsto \beta_{N}(\epsilon, m ) - f $.

On the other hand, the extension $\delta \alpha( \epsilon ) $ is the kernel of the following
subdiagram of (\ref{hexp2})
\begin{equation}\label{hag}
\begin{CD}
	 M\{t\}^{\sigma\tau} \oplus k[t] @>>> 0  \\
 	            @VVV                            @VVV \\
	 M\llb t-\th \rrb \oplus C\llp t-\th \rrp    @>>>     C\llp t-\th \rrp/C\llb t-\th \rrb.                     
\end{CD}
\end{equation}
To conclude that the extensions $ \delta \alpha( \epsilon )  $ and $ h\exp(\epsilon) $
coincide it suffices to observe that the quotient of the diagram (\ref{hexp2}) by the diagram
(\ref{hag}) is an isomorphism when seen as a map from the left column to the right column.
\end{proof}

\begin{lemma}
	The homomorphism $ h: E( C ) \to \Ext^1( H, \One ) $ is
	locally analytic.
\end{lemma}

\begin{proof}[Proof of the lemma]
Denote the rank of $ M $ by $ r $  and choose a $ K[t] $\dash basis of $ M $. This induces a basis of $ M' $, so let the $r$ by $r$ matrix $ A $ over $ C[t] $ represent the linear map
$ \sigma : M'_C \to M_C $ with respect to these bases. Extensions of $ M_C $ by $ C[t] $ are then represented by block matrices of the form
\[
	\left(
	\begin{array}{cc}
		1 & B \\
		0 & A 
	\end{array}
	\right)
\]
where $ B $ is a length $r$ row vector. Denote the extension of $ M_C $ by $ C[t] $ given by
$ B $ by $ \tilde{M}(B) $. There is an integer $ d $ such that all extension classes are represented by some $ B $ whose entries have degree at most $ d $ and it suffices to show that
the homomorphism $ h $ is locally analytic as a function of the coefficients of these entries.

The $\sigma\tau$\dash invariants of $ \tilde{M}(B)\{t\} $ are those pairs
$ (v, w ) $ with $ v \in C\{t\} $ and $ w \in M\{t\}^{\sigma\tau} $ such that
\[
	\tau(v) - v = B \cdot \tau(w).
\]
Fix a basis $ w_1, \ldots, w_r $ of $ M\{t\}^{\sigma\tau} $.
Now there exists an $\epsilon>0$ such that for all $B$ whose entries have degree at most $ d $
and whose coefficients have absolute value at most $ \epsilon $ and for all $i$ the power series 
$ B\cdot \tau(w_i) \in C\{t\} $ has all coefficients of absolute value smaller than $1$.

For these $ B $, let $ v_i $ be the unique element of $ C\{t\} $ such that
$ \tau(v_i) - v_i = B \cdot \tau(w_i) $ and which has all coefficients smaller than $1$ in absolute value. This $v_i$ is given by the formula
\[
	v_i = \sum_{j=0}^\infty \tau^j( - B )
\]
and clearly is an analytic function of $ B $. The $ (v_i, w_i ) $ together with $ (1,0) $ 
form a basis of $ \tilde{M}(B)\{t\}^{\sigma\tau} $ and using this basis one sees that
the resulting extension of Hodge structures depends analytically on the coefficients of $ B $. 
\end{proof}

\begin{lemma}\label{locan}
	There exists a neighbourhood $ U $ of $ 0 $ in $ \Lie_E( C ) $ on which
	$ h \circ \exp_E $ and $ u \circ \exp_E $ agree.
\end{lemma}

\begin{proof}
Since $ h \circ \exp_E $ is locally analytic and $ u \circ \exp_E $ is analytic their difference 
$ f $ is a locally analytic function on $ \Lie_E( C ) $. 

Now by Proposition \ref{twohodges} we have $ \Lie_E( C ) \cong  C\llb t-\th \rrb^r / L $ with $ L $ 
a free sub $ C\llb t-\th \rrb $\dash module of rank $ r $, such that $ \Lambda_E $ is the image of the natural map $ k[t]^r \to C\llb t-\th \rrb^r / L $. By Lemma \ref{torsion} $ f $ vanishes on the image of 
\[
	k\llp t^{-1} \rrp^r \to C\llb t-\th \rrb^r / L.
\]
We will show that any locally analytic additive function on $ C\llb t-\th \rrb^r / L $
that vanishes on the image of $ k\llp t^{-1} \rrp^r $ is locally constant, and so in particular that there exists a neighbourhood $ U $ of $ 0 $ on which $ f $ vanishes.

Clearly enlarging $ L $ we may assume that $ L = (t-\th)^s C\llb t-\th \rrb^r $ for some positive integer $ s $.  Also, without loss of generality we may assume that $ r = 1 $.  So it suffices to show that  any locally analytic additive function $ g $ on the $ C $-vector space 
\[
	V = C\llb t-\th \rrb / (t-\th)^s C\llb t-\th \rrb
\]
that vanishes on the image of $ k\llp t^{-1} \rrp $ is locally constant.

So assume $ g : V \to C $ is not locally constant. Then the kernel of $ g $ is a smooth analytic subgroup of $ V $ whose tangent space $ T \subset V $ to the origin is a proper sub-vector space.

For any positive integer $ i $ denote the image of $ t^{-i} \in C\llb t-\th \rrb $ in
$ V $ by $ v_i $. The $ v_i $ converge to $ 0 $ as $ i $ tends to infinity. Let $ [v_i] $ denote
the image of $ v_i $ in the projective space $ \PP( V ) $.  The sequence $ ([v_i])_i $ is
periodic yet it there is no  proper subspace $ T \subset V $ such that 
the sequence lies within $ \PP( T ) \subset \PP( V ) $, a contradiction.
\end{proof}

\begin{proof}[Proof of Proposition \ref{pextext}]
Let $ \epsilon $ be an element of $ \Lie_E( C ) $. Note that for sufficiently large $ n $
we have that the actions of $ t^{p^n} $ and $ \th^{p^n} $ on $ \Lie_E( C ) $ coincide. Possibly making $ n $ even larger we may also assume that $ \th^{-p^n}\! \epsilon $ lies in the $ U $ whose existence is asserted in Lemma \ref{locan}. Using such $ n$ we find 
\[
	 (h-u)\exp_E(\epsilon) = t^{p^n}(h-u)\exp_E(\th^{-p^n}\!\epsilon) = 0.
\]
This finishes the proofs of Proposition \ref{pextext} and of Theorem \ref{extext}.
\end{proof}

\bibliographystyle{plain}
\bibliography{../master}

\def\cprime{$'$} \def\cprime{$'$} \def\cprime{$'$} \def\cprime{$'$}
  \def\cprime{$'$} \def\cprime{$'$} \def\cprime{$'$}
\begin{thebibliography}{10}

\bibitem{Anderson86}
Greg~W. Anderson.
\newblock $t$-motives.
\newblock {\em Duke Math. J.}, 53(2):457--502, 1986.
\newblock \href{http://www.ams.org/mathscinet-getitem?mr=850546}{MR850546}.

\bibitem{Anderson04}
Greg~W. Anderson, W.~Dale Brownawell, and Matthew~A. Papanikolas.
\newblock Determination of the algebraic relations among special
  {$\Gamma$}-values in positive characteristic.
\newblock {\em Ann. of Math. (2)}, 160(1):237--313, 2004.
\newblock \href{http://www.ams.org/mathscinet-getitem?mr=2119721}{MR2119721}.

\bibitem{Anderson90}
Greg~W. Anderson and Dinesh~S. Thakur.
\newblock Tensor powers of the {C}arlitz module and zeta values.
\newblock {\em Ann. of Math. (2)}, 132(1):159--191, 1990.
\newblock \href{http://www.ams.org/mathscinet-getitem?mr=1059938}{MR1059938}.

\bibitem{Deligne74b}
Pierre Deligne.
\newblock Th\'eorie de {H}odge. {III}.
\newblock {\em Inst. Hautes \'Etudes Sci. Publ. Math.}, (44):5--77, 1974.
\newblock \href{http://www.ams.org/mathscinet-getitem?mr=0498552}{MR0498552}.

\bibitem{Dieudonne57}
Jean Dieudonn{\'e}.
\newblock Groupes de {L}ie et hyperalg\`ebres de {L}ie sur un corps de
  caract\'eristique {$p>0$}. {VII}.
\newblock {\em Math. Ann.}, 134:114--133, 1957.
\newblock \href{http://www.ams.org/mathscinet-getitem?mr=0098146}{MR0098146}.

\bibitem{Drinfeld74E}
V.~G. Drinfel{\cprime}d.
\newblock Elliptic modules.
\newblock {\em Mat. Sb. (N.S.)}, 94(136):594--627, 656, 1974.
\newblock \href{http://www.ams.org/mathscinet-getitem?mr=0384707}{MR0384707}.

\bibitem{Gardeyn01}
Francis Gardeyn.
\newblock {\em $t$-Motives and Galois Representations}.
\newblock PhD thesis, Universiteit Gent, 2001.

\bibitem{Laumon96}
G{\'e}rard Laumon.
\newblock {\em Cohomology of {D}rinfeld modular varieties. {P}art {I}},
  volume~41 of {\em Cambridge Studies in Advanced Mathematics}.
\newblock Cambridge University Press, Cambridge, 1996.
\newblock \href{http://www.ams.org/mathscinet-getitem?mr=1381898}{MR1381898}.

\bibitem{Papanikolas03}
Matthew~A. Papanikolas and Niranjan Ramachandran.
\newblock A {W}eil-{B}arsotti formula for {D}rinfeld modules.
\newblock {\em J. Number Theory}, 98(2):407--431, 2003.
\newblock \href{http://www.ams.org/mathscinet-getitem?mr=1955425}{MR1955425}.

\bibitem{Pink97}
Richard Pink.
\newblock Hodge structures over function fields.
\newblock {\em pre-print}, 1997.

\bibitem{Serre68}
Jean-Pierre Serre.
\newblock {\em Abelian {$l$}-adic representations and elliptic curves}.
\newblock McGill University lecture notes written with the collaboration of
  Willem Kuyk and John Labute. W. A. Benjamin, Inc., New York-Amsterdam, 1968.
\newblock \href{http://www.ams.org/mathscinet-getitem?mr=0263823}{MR0263823}.

\bibitem{Stalder07}
N.~R. Stalder.
\newblock {\em Algebraic Monodromy Groups of $A$-Motives}.
\newblock PhD thesis, ETH Z\"urich, 2007.

\bibitem{Taelman09a}
L.~Taelman.
\newblock Artin {$t$}-{Motifs}.
\newblock {\em J. Number Theory}, 129:142--157, 2009.
\newblock \href{http://www.ams.org/mathscinet-getitem?mr=2468475}{MR2468475}.

\bibitem{Taelman09b}
Lenny Taelman.
\newblock Special {$L$}-values of {$t$}-motives: a conjecture.
\newblock {\em Int. Math. Res. Not. IMRN}, (16):2957--2977, 2009.
\newblock \href{http://www.ams.org/mathscinet-getitem?mr=2533793}{MR2533793}.

\bibitem{Woo95}
Sung~Sik Woo.
\newblock Extensions of {D}rinfel\cprime d modules of rank {$2$} by the
  {C}arlitz module.
\newblock {\em Bull. Korean Math. Soc.}, 32(2):251--257, 1995.
\newblock \href{http://www.ams.org/mathscinet-getitem?mr=1356079}{MR1356079}.

\end{thebibliography}

\end{document}